\documentclass{amsart}

\input xy
\xyoption{all}

\newcommand{\A}{\mathcal A}
\newcommand{\BB}{\mathcal B}
\newcommand{\DD}{\mathcal D}
\newcommand{\C}{\mathcal C}
\newcommand{\K}{\mathcal K}
\newcommand{\II}{\mathbb I}
\newcommand{\U}{\mathcal U}
\newcommand{\V}{\mathcal V}
\newcommand{\W}{\mathcal W}
\newcommand{\Z}{\mathcal Z}
\newcommand{\F}{\mathcal F}
\newcommand{\St}{{\mathcal S}t}
\newcommand{\w}{\omega}
\newcommand{\ulim}{\textstyle{\bigcup^\infty}}
\newcommand{\id}{\mathrm{id}}
\newcommand{\pr}{\mathrm{pr}}
\newcommand{\IN}{\mathbb N}
\newcommand{\ANE[1]}{\mathrm{ANE}[{#1}]}
\newcommand{\AEE[1]}{\mathrm{AE}[{#1}]}
\newcommand{\e}{\varepsilon}
\newcommand{\dist}{\mathrm{dist}}
\newcommand{\diam}{\mathrm{diam}}

\newtheorem{theorem}{Theorem}[section]
\newtheorem{proposition}[theorem]{Proposition}
\newtheorem{problem}[theorem]{Problem}
\newtheorem{corollary}[theorem]{Corollary}
\newtheorem{lemma}[theorem]{Lemma}
\newtheorem{claim}[theorem]{Claim}
\theoremstyle{definition}
\newtheorem{definition}[theorem]{Definition}

\title{Universal nowhere dense and meager sets in Menger manifolds}
\author{Taras Banakh and Du\v san Repov\v s}
\subjclass{??}
\address{T.Banakh: Jan Kochanowski University in Kielce (Poland) and Ivan Franko National University of Lviv (Ukraine)}
\email{t.o.banakh@gmail.com}
\address{D.Repov\v s: Faculty of  Education, and
Faculty of Mathematics and Physics, University of Ljubljana (Slovenia)}
\email{dusan.repovs@guest.arnes.si}
\thanks{This research was supported by the Slovenian Research Agency grants P1-0292-0101 and J1-4144-0101. The first author has been partially financed by NCN means granted by decision DEC-2011/01/B/ST1/01439.}

\begin{document}
\begin{abstract} In each Menger manifold $M$ we construct:
\begin{itemize}
\item a closed nowhere dense subset $M_0$ which is homeomorphic to $M$ and is universal nowhere dense in the sense that for each nowhere dense set $A\subset M$ there is a homeomorphism $h$ of $M$ such that $h(A)\subset M_0$;
\item a meager $F_\sigma$-set $\Sigma_0\subset M$ which is universal meager in the sense that for each meager subset $B\subset M$ there is a homeomorphism $h$ of $M$ such that $h(B)\subset \Sigma_0$.
\end{itemize}
Also we prove that any two universal meager $F_\sigma$-sets in $M$ are ambiently homeomorphic.
\end{abstract}

\maketitle

\section{Introduction and survey of principal results}

In this paper we shall construct universal nowhere dense and universal meager sets in Menger manifolds, i.e., manifolds modeled on Menger cubes $\mu^n$, $n\ge 0$. The notions of a universal nowhere dense or universal meager set are special cases of the notion of a $\K$-universal set for a family $\K$ of subsets of a topological space $X$.
A subset $A\subset X$ is called {\em $\K$-universal} if $A\in\K$ and for each $B\in\K$ there exists a homeomorphism $h:X\to X$ such that $h(B)\subset A$.

\begin{problem} Which families $\K$ of subsets of a given topological space $X$ do possess a $\K$-universal set $A\in\K$?
\end{problem}

For certain families $\K$ this problem is closely connected with the problem of the existence of $\K$-absorptive sets, whose definition we are going to recall now.

Let $X$ be a topological space and $\K$ is a family of subsets of $X$. We shall assume that $\K$ is {\em topologically invariant} in the sense that $\K=\{h(K):K\in\K\}$ for any homeomorphism $h:X\to X$ of $X$. By $\sigma\K$ we denote the family of subsets of $X$ which can be written as countable unions $A=\bigcup_{n\in\w}A_n$ of subsets $A_n\in\K$, $n\in\w$. For two maps $f,g:X\to Y$ between topological spaces and an open cover $\U$ of $Y$ we shall write $(f,g)\prec\U$ and say that the maps $f,g$ are {\em $\U$-near} if for each $x\in X$ the set $\{f(x),g(x)\}$ is contained in some set $U\in\U$.

Following \cite{West}, we define a subset $B\subset X$ to be {\em $\K$-absorptive} in $X$ if $B\in\sigma\K$ and for each set $K\subset\K$, open set $V\subset X$, and open cover $\U$ of $V$, there is a homeomorphism $h:V\to V$ such that $h(K\cap V)\subset B\cap V$ and $(h,\id)\prec\U$. An important observation is that each set $A\in\sigma\K$ containing a $\K$-absorptive subset of $X$ is also $\K$-absorptive.

The following powerful uniqueness theorem was proved by West \cite{West} and Geoghegan and Summerhill  \cite[2.5]{GS74}.

\begin{theorem}[Uniqueness Theorem for $\K$-absorptive sets]\label{t1} Let $\K$ be a topologically invariant family of closed subsets of a Polish space $X$. Then any two $\K$-absorptive sets $B,B'\subset X$ are ambiently homeomorphic. More precisely, for any open set $V\subset X$ and any open cover $\V$ of $V$ there is a homeomorphism $h:V\to V$ such that $h(V\cap B)=V\cap B'$ and $h$ is $\V$-near to the identity map of $V$.
\end{theorem}

Two subsets $A,B$ of a topological space $X$ are called {\em ambiently homeomorphic} if there is a homeomorphism $h:X\to X$ such that $h(A)=B$. This happens if and only if the pairs $(X,A)$ and $(X,B)$ are homeomorphic. We shall say that two pairs $(X,A)$ and $(Y,B)$ of topological spaces $A\subset X$ and $B\subset Y$ are {\em homeomorphic} if there is a homeomorphism $h:X\to Y$ such that $h(A)=B$. In this case we say that $h:(X,A)\to (Y,B)$ is a homeomorphism of pairs.

As shown in Corollary 1 of \cite{BR2}, Theorem~\ref{t1} implies the following characterization.

\begin{theorem}\label{c1} Let $\K$ be a topologically invariant family of closed subsets of a Polish space. If a $\K$-absorptive set $B$ in $X$ exists, then a subset $A\subset X$ is $\sigma\K$-universal in $X$ if and only if $A$ is $\K$-absorptive.
\end{theorem}

Theorem~\ref{c1} reduces the problem of constructing $\sigma\K$-universal sets in a Polish space $X$ to the problem of constructing a $\K$-absorptive sets in $X$. The latter problem was extensively studied for various families $\K$ consisting of $Z$-sets in $X$, see \cite[Ch.Iv,V]{BP}, \cite{BRZ}, \cite[\S2.2.2]{Chig}.

Let us recall that a subset $A$ of a topological space $X$ is a {\em $Z$-set} in $X$ if $A$ is closed in $X$ and for each open cover $\U$ of $X$ there is a map $f:X\to X\setminus A$ such that $(f,\id)\prec\U$. The family $\Z$ of all $Z$-sets of $X$ is topologically invariant and is contained in the larger families $\Z_k$, $0\le k\le \w$, consisting of $Z_k$-sets in $X$. A subset $A$ of a topological space $X$ is called a {\em $Z_k$-set} in $X$ for $0\le k\le\w$ if $A$ is closed in $X$ and for each open cover $\U$ of $X$ and each  map $f:\II^k\to X$ there is a map  $f':\II^k\to X\setminus A$ such that $(f',f)\prec\U$. It is clear that $\Z_k\subset\Z_n$ for any numbers $0\le n\le k\le \w$, which implies that the families $\Z,\Z_k$, $\w\ge k\ge 0$, form an increasing chain
$$\Z\subset\Z_\w\subset\cdots\subset\Z_1\subset\Z_0.$$
For certain nice spaces, for example, ANR's the families $\Z$ and $\Z_\w$ coincide (see \cite[2.2.4]{Chig}).
On the other hand, for any topological space $X$ the family $\Z_0$ of $Z_0$-sets coincides with the family of closed nowhere dense subsets of $X$. Consequently, the family $\sigma\Z_0$ coincides with the family of all meager $F_\sigma$-subsets of $X$.

Using the technique of skeletoids, $\Z$-absorptive sets (which are automatically $\sigma\Z$-universal) were constructed in many ``nice'' spaces $X$, in particular, in manifolds modeled on the Hilbert or Menger cubes, see \cite[2.2.2]{Chig}. By a {\em manifold modeled on a topological space $E$} (briefly, an {\em $E$-manifold}) we understand a paracompact topological space having a cover by open subsets homeomorphic to open subsets of the model space $E$. The standard technique of skeletoids cannot be applied to constructing $\K$-universal or $\sigma\K$-universal sets for families  $\K\not\subset\Z$. In \cite{BR1} and \cite{BR2} the authors using the technique of tame open set and tame $G_\delta$-set, constructed $\Z_0$-universal and $\sigma\Z_0$-universal sets in manifolds modeled on finite-dimensional and infinite-dimensional cubes $\II^n$, $1\le n\le\w$.

In this paper we shall apply the same technique to construct $\Z_0$-universal sets and $\sigma\Z_0$-universal sets in Menger manifolds, i.e., manifolds modeled on Menger cubes $\mu^n$, $0\le n<\w$. There are many (topologically equivalent) constructions of Menger cubes \cite[\S4.1.1]{Chig}. Due to the celebrated Bestvina's characterization \cite{Best}, spaces homeomorphic to Menger cubes can be topologically characterized as follows.

\begin{theorem}[Bestvina]\label{best1} A compact metrizable space $X$ is homeomorphic to the $n$-dimensional Menger cube $\mu^n$ if any only if
\begin{enumerate}
\item $\dim(X)=n$;
\item $X$ is an absolute extensor in dimension $n$;
\item $X$ has disjoint $n$-cells property.
\end{enumerate}
\end{theorem}

We say that a topological space $X$ has {\em disjoint $n$-cells property} if for any maps $f,g:\II^n\to X$ from the $n$-dimensional cube $\II^n=[0,1]^n$ to $X$ and any open cover $\U$ of $X$ there are maps $f',g':\II^n\to X$ such that $f'(\II^n)\cap g'(\II^n)=\emptyset$ and the maps $f',g'$ are $\U$-near to the maps $f,g$, respectively.

A topological space $X$ is called an {\em absolute neighborhood extensor in dimension $n$} (briefly, an {\em $\ANE[n]$-space}) if each continuous map $f:B\to X$ defined on a closed subset $B$ of a metrizable space $A$ of dimension $\dim(A)\le n$ has a continuous extension $\bar f:O(B)\to X$ defined on a neighborhood $O(B)$ of the set $B$ in $A$. If $f$ can be always extended to a continuous map $\bar f:A\to X$, then we say that $X$ is an {\em absolute extensor in dimension $n$} (briefly, an {\em $\AEE[n]$-space}).

The ``local'' version of Theorem~\ref{best1} yields a characterization of Menger manifolds, see \cite{Best} or \cite[4.1.9]{Chig}.

 \begin{theorem}[Bestvina]\label{best2} A locally compact metrizable space $X$ is a Menger manifold if any only if
\begin{enumerate}
\item $X$ has finite dimension $n=\dim(X)$;
\item $X$ is an absolute neighborhood extensor for $n$-dimensional spaces;
\item $X$ has disjoint $n$-cells property.
\end{enumerate}
\end{theorem}

For infinite $n$, Theorems~\ref{best1} and \ref{best2} turn into Toru\'nczyk's characterization
theorems for the Hilbert cube and Hilbert cube manifolds, see \cite{Tor80}.

In this paper we address the following:

\begin{problem}\label{prob1.6} Is it true that for every $k,n\in\w$ the Menger cube $\mu^n$ contains a universal $\Z_k$-set and a universal $\sigma\Z_k$-set?
\end{problem}

Z-Set Unknotting Theorem \cite[4.1.16]{Chig} and Theorem 5.2.1 \cite{Chig} on the existence of skeletoids in the Menger cube imply an affirmative answer to this problem for all $k\ge n$.
In this paper we answer Problem~\ref{prob1.6} for $k=0$.
To construct $\Z_0$-universal and $\sigma\Z_0$-universal sets in Menger cubes (or more generally in Menger manifolds), we shall use the technique of tame open sets and tame $G_\delta$-sets developed in \cite{BR2}.

We start by defining tame open balls in Menger manifolds. Let $M$ be a $\mu^n$-manifold for some $n\in\w$.

An open subset $B\subset M$ is called a {\em tame open ball\/} in $M$ if
\begin{itemize}
\item the closure $\bar B$ of $B$ is homeomorphic to the Menger cube $\mu^n$;
\item the boundary $\partial B=\bar B\setminus B$ in $M$ is homeomorphic to the Menger cube $\mu^n$;
\item $\partial B$ is a $Z$-set in $\bar B$ and in $M\setminus B$.
\end{itemize}
By a {\em tame closed  ball} in $M$ we shall understand the closure of a tame open ball in $M$.

The following theorem follows from Lemma~\ref{l4.2}, which will be proved in Section~\ref{s4}.

\begin{theorem}\label{base} Each Menger manifold has a base of topology consisting of tame open balls.
\end{theorem}

A family $\F$ of subsets of a topological space $X$ is called {\em vanishing} if for each open cover $\U$ of $X$ the family $\F'=\{F\in\F:\forall U\in\U,\;\;F\not\subset U\}$ is locally finite in $X$. It is easy to see that a countable family $\F=\{F_n\}_{n\in\w}$ of subsets of a compact metric space $(X,d)$ is vanishing if and only if $\lim_{n\to\infty}\mathrm{diam}(F_n)=0$.

A subset $U$ of a Menger manifold $M$ is called a {\em tame open set in $M$} if $U=\bigcup\U$ for some vanishing family $\U$ of tame open balls having pairwise disjoint closures in $M$. If $n=\dim(M)>0$, then the tame open balls $U\in\U$ can be uniquely recovered as connected components of the tame open set $U$.

A subset $G\subset X$ is called a {\em tame $G_\delta$-set in $X$} if $U=\bigcap_{n\in\w}\bigcup \U_n$ for some families $\U_n$, $n\in\w$, of tame open balls in $X$ such that
\begin{itemize}
\item for every $n\in\w$ any distinct tame open balls $U,V\in \U_n$ have disjoint closures in $X$;
\item for every $n\in\w$ and $U\in\U_{n+1}$ the closure $\bar U$ is contained in some set $V\in\U_n$ with $V\ne\bar U$;
\item the family $\bigcup_{n\in\w}\U_n$ is vanishing.
\end{itemize}

Tame open sets and tame $G_\delta$-sets  can be equivalently defined via tame families of tame open balls.
A family $\U$ of non-empty open subsets of a topological space $X$ is called {\em tame} if $\U$ is vanishing and for any distinct sets $U,V\in\U$ one of three possibilities hold: either $\bar U\cap\bar V=\emptyset$ or $\bar U\subset V$ or $\bar V\subset U$.
For a family $\U$ of subsets of a set $X$ by $$\ulim\U=\textstyle{\bigcap\big\{\bigcup}(\U\setminus\F):\F\mbox{ is a finite subfamily of $\U$}\big\}$$ we denote the set of all points $x\in X$ which belong to infinite number of sets $U\in\U$.

The following proposition can be proved by analogy with Proposition 2 of \cite{BR2}.

\begin{proposition}\label{p2} A subset $T$ of a Menger manifold $M$ is tame open \textup{(}resp. tame $G_\delta$\textup{)} if and only if $T=\bigcup\mathcal T$ \textup{(}resp. $T=\ulim\mathcal T$~\textup{)} for a suitable tame family $\mathcal T$ of tame open balls in $M$.
\end{proposition}

The classes of dense tame open sets and dense tame $G_\delta$-sets in Menger manifolds have the following cofinality property, which can be derived from Theorem~\ref{base} by analogy with the proof of Proposition~3 of \cite{BR2}.

\begin{proposition}\label{p3}
\begin{enumerate}
\item Each open subset of a Menger manifold contains a dense tame open set;
\item Each $G_\delta$-subset of a Menger manifold contains a dense tame $G_\delta$-set.
\end{enumerate}
\end{proposition}

The $Z$-Set Unknotting Theorem for Menger manifolds \cite[4.1.15]{Chig} implies that any two tame open balls in a connected Menger manifold are ambiently homeomorphic. A similar uniqueness theorem holds also for dense tame open sets and dense tame $G_\delta$-sets in Menger manifolds.

\begin{theorem}[Uniqueness Theorem for Dense Tame Open Sets]\label{tameopen} Any two dense tame open sets $U,U'\subset M$ of a Menger manifold $M$ are ambiently homeomorphic.
Moreover, for each tame open set $U\subset M$ its complement $M\setminus U$ is homeomorphic to the Menger manifold $M$.
\end{theorem}

This theorem follows from Propositions~\ref{p3.6} and \ref{p4.3}, which will be proved in Sections~\ref{s3} and \ref{s4}, respectively. Using Theorem~\ref{tameopen} and Proposition~\ref{p3}(1) we can give a promised partial answer to Problem~\ref{prob1.6}.

\begin{theorem} Each Menger manifold $M$ contains a $\Z_0$-universal subset $A\subset M$, which is homeomorphic to $M$.
\end{theorem}

\begin{proof} By Proposition~\ref{p3}, the Menger manifold $M$ contains a dense tame open set $U\subset M$. By Theorem~\ref{tameopen}, the complement $A=M\setminus U$ is a closed nowhere dense subset of $M$, homeomorphic to $M$. To prove that the set $A$ is $\Z_0$-universal, take any closed nowhere dense subset $B\subset X$. By Proposition~\ref{p3}, the dense open set $X\setminus B$ contains a dense tame open set $V\subset M$. By Theorem~\ref{tameopen}, there exists a homeomorphism $h:M\to M$ such that $h(V)=U$ and hence $h(B)\subset h(M\setminus V)=M\setminus U=A$. This proves that the set $A$ is $\Z_0$-universal.
\end{proof}

Next, we turn to universal $\sigma\Z_0$-sets in Menger manifolds. We shall exploit the following uniqueness theorem, which will be proved in Section~\ref{s5}.

\begin{theorem}[Uniqueness Theorem for Dense Tame $G_\delta$-Sets]\label{t3} Any two dense tame $G_\delta$-sets $G,G'$ in a Menger manifold $M$ are ambiently homeomorphic. Moreover, for each open cover $\U$ of $M$ there is a homeomorphism $h:(M,G)\to(M,G')$, which is $\U$-near to the identity homeomorphism of $M$.
\end{theorem}

Applying Theorem~\ref{t3} and Proposition~\ref{p3}(2) we can give another partial answer to Problem~\ref{prob1.6}.

\begin{theorem}[Characterization of $\sigma\Z_0$-Universal Sets in Menger Manifolds]\label{t4} For a subset $A$ of a Menger manifold $X$ the following conditions are equivalent:
\begin{enumerate}
\item $A$ is $\sigma\Z_0$-universal in $X$;
\item $A$ is $\Z_0$-absorptive in $X$;
\item the complement $X\setminus A$ is a dense tame $G_\delta$-set in $X$.
\end{enumerate}
\end{theorem}

The proof of this theorem literally repeats the proof of Theorem~4 of \cite{BR2} characterizing $\sigma\Z_0$-universal sets in manifolds modeled on the Hilbert cube.

Taking into account that each meager $F_\sigma$-set containing a $\sigma\Z_0$-absorptive subset is $\sigma\Z_0$-absorptive, we see that Theorem~\ref{t4} implies:

\begin{corollary}  Each dense $G_\delta$-subset of a dense tame $G_\delta$-set in a Menger manifold is tame.
\end{corollary}

\section{Equivalence of certain decompositions of Polish spaces}

Theorem~\ref{tameopen} will be derived from results \cite{BR1} on the topological equivalence of upper semicontinuous decompositions of a given Polish space. To use these results we need to recall some terminology from \cite{BR1}. First we note that all {\em maps} considered in this paper are continuous.

Let $\A,\BB$ be two families $\A,\BB$ of subsets of a space $X$. We shall write $\A\prec\BB$ and say that the family $\A$ {\em refines} the family $\BB$ if each set $A\in\A$ is contained in some set $B\in\BB$.

A subset $A\subset X$ is called {\em $\BB$-saturated} if $A$ coincides with its {\em $\BB$-star} $\St(A,\BB)=\bigcup\{B\in\BB:A\cap B\ne \emptyset\}$. The family $\A$ is called {\em $\BB$-saturated} if each set $A\in\A$ is $\BB$-saturated. The family $\St(\A,\BB)=\{\St(A,\BB):A\in\BB\}$ will be called the {\em $\BB$-star} of the family $\A$, and $\St(\A)=\St(\A,\A)$ is the {\em star} of $\A$.

By a {\em decomposition} of a topological space $X$ we understand a cover $\DD$ of $X$ by pairwise disjoint non-empty compact subsets. For each decomposition $\DD$ we can consider the quotient map $q_\DD:X\to \DD$ assigning to each point $x\in X$ the unique compact set $q(x)\in\DD$ that contains $x$. The quotient map $q_\DD$ induces the quotient topology on $\DD$ turning $\DD$ into a topological space called the {\em decomposition space} of the decomposition $\DD$. Sometimes to distinguish a decomposition $\DD$ from its decomposition space we shall denote the latter space by $X/\DD$.

A decomposition $\DD$ of a topological space $X$ is said to be {\em upper semicontinuous} if for each closed subset $F\subset X$ its {\em $\DD$-saturation} $\St(F,\DD)=\bigcup\{D\in\DD:D\cap F\ne\emptyset\}$ is closed in $X$. It is easy to see that a decomposition $\DD$ of $X$ is upper semicontinuous if and only if the quotient map $q_\DD:X\to X/\DD$ is closed if and only if the quotient map $q_\DD$ is perfect (the latter means that $q_\DD$ is closed and for each point $y\in X/\DD$ the preimage $q_\DD^{-1}(y)$ is compact).
Since the metrizability (and local compactness) of spaces is preserved by perfect maps (see \cite[3.7.21]{En} and \cite[4.4.15]{En}), we get the following:

\begin{lemma}\label{l2.1} For every upper semicontinuous decomposition $\mathcal D$ of a metrizable (locally compact) space $X$ the decomposition space $X/\mathcal D$ is metrizable (and locally compact).
\end{lemma}

Observe that a decomposition $\DD$ of a topological space $X$ is {\em vanishing} if and only if for each open cover $\U$ of $X$ the subfamily $\DD'=\{D\in\DD:\forall U\in\U\;\;D\not\subset U\}$ is discrete in $X$ in the sense that each point $x\in X$ has a neighborhood $O_x\subset X$ that meets at most one set $D\in\DD'$.

Each  vanishing disjoint family $\C$ of non-empty compact subsets of a topological space $X$ generates the  vanishing decomposition
$$\dot\C=\C\cup\big\{\{x\}:x\in X\setminus\textstyle{\bigcup}\C\big\}$$of the space $X$.
In particular, each non-empty compact set $K\subset X$ induces the vanishing decomposition $\{K\}\cup\big\{\{x\}:x\in X\setminus K\}$ whose decomposition space will be denoted by $X/K$.
By $q_K:X\to X/K$ we shall denote the corresponding quotient map. By Lemma~2.2 of \cite{BR1}, each vanishing decomposition $\DD$ of a regular space $X$ is upper semicontinuous.

A decomposition $\DD$ of a space $X$ will be called {\em dense} if its {\em non-degeneracy part} $$\DD^\circ=\{D\in\DD:|D|>1\}$$ is dense in the decomposition space $\DD=X/\DD$.

A decomposition $\DD$ of a topological space $X$ is called
\begin{itemize}
\item {\em shrinkable} if for each $\DD$-saturated open cover $\U$ of $X$ and each open cover $\V$ of $X$ there is a homeomorphism $h:X\to X$ such that $(h,\id_X)\prec\U$ and $\{h(D):D\in\DD\}\prec\V$;
\item {\em strongly shrinkable} if for each $\DD$-saturated open set $U\subset X$ the decomposition $\DD|U=\{D\in\DD:D\subset U\}$ of $U$ is shrinkable.
\end{itemize}
A compact subset $K$ of a topological space $X$ is called {\em locally shrinkable}  if for each neighborhood $O(K)\subset X$ and any open cover $\V$ of $O(K)$ there is a homeomorphism $h:X\to X$ such that $h|X\setminus O(K)=\id$ and $h(K)$ is contained in some set $V\in\V$. It is easy to see that a compact subset $K\subset X$ is locally shrinkable if and only if  the decomposition $\{K\}\cup\big\{\{x\}:x\in X\setminus K\big\}$ of $X$ is strongly shrinkable (cf. \cite[p.42]{Dav}).

(Strongly) shrinkable decompositions are closely connected with (strong) near homeomorphisms.

A map $f:X\to Y$ between topological spaces will be called a
\begin{itemize}
\item a {\em near homeomorphism} if for each open cover $\U$ of $Y$ there is a homeomorphism $h:X\to Y$ such that $(h,f)\prec\U$;
\item a {\em strong near homeomorphism} if for each open set $U\subset Y$ the map $f|f^{-1}(U):f^{-1}(U)\to U$ is a near homeomorphism.
\end{itemize}

The following Shrinkability Criterion was proved in \cite[Theorem 2.6]{Dav}.

\begin{theorem}[Shrinkability Criterion]\label{t:shrink} An upper semicontinuous decomposition $\DD$ of a completely metrizable space $X$ is (strongly) shrinkable if and only if the quotient map $q_\DD:X\to X/\DD$ is a (strong) near homeomorphism.
\end{theorem}

We shall say that a decomposition $\A$ of a topological space $X$ is {\em topologically equivalent} to a decomposition $\BB$ of a topological space $Y$ if there is a homeomorphism $\Phi:X\to Y$ such that the
decomposition $\Phi(\A)=\{\Phi(A):A\in\A\}$ of $Y$ is equal to the decomposition $\BB$.

Now we shall formulate some conditions of topological equivalence of decompositions of Polish spaces.
First we introduce two definitions from \cite{BR1}.

\begin{definition}\label{d:K-tame} Let $\K$ be a family of compact subsets of a topological space $X$.
We shall say that the family $\K$
\begin{itemize}
\item is {\em topologically invariant} if for each homeomorphism $h:X\to X$ and each set $K\in\K$ we get $h(K)\in\K$;
\item has the {\em local shift property} if for any point $x\in X$ and a neighborhood $O_x\subset X$ there is a neighborhood $U_x\subset O_x$ of $x$ such that for any sets $A,B\in\K$ with $A,B\subset U_x$ there is a homeomorphism $h:X\to X$ such that $h(A)=B$ and $h|X\setminus O_x=\id|X\setminus O_x$;
\item {\em tame} if $\K$ is topologically invariant, consists of locally shrinkable sets, has the local shift property, and each non-empty open subset $U\subset X$ contains a set $K\in\K$.
\end{itemize}
\end{definition}

Now we can define $\K$-tame decompositions.

\begin{definition}\label{d:d-Ktame} Let $\K$ be a tame family of compact subsets of a Polish space $X$. A decomposition $\DD$ of $X$ is called {\em $\K$-tame} if $\DD$ is vanishing, strongly shrinkable, and $\DD^\circ\subset\K$.
\end{definition}

The following theorem, proved in \cite[2.6]{BR1}, yields many examples of $\K$-tame decompositions.

\begin{theorem}\label{t:exist} Let $\K$ be a tame family of compact subsets of a completely metrizable space $X$ such that each set $K\in\K$ contains more than one point. For any open set $U\subset X$ there is a $\K$-tame decomposition $\DD$ of $X$ such that $\bigcup\DD^\circ$ is a dense subset of $U$.
\end{theorem}

We shall say that a topological space $X$ is {\em strongly locally homogeneous} if the family of singletons $\big\{\{x\}\big\}_{x\in X}$ is tame. This happens if and only if this family has the local shift property. So, our definition of the strong local homogeneity agrees with the classical one introduced in \cite{Bennett}. It is easy to see that each connected strongly locally homogeneous space is {\em topologically homogeneous} in the sense that for any two points $x,y\in X$ there is a homeomorphism $h:X\to X$ with $h(x)=y$.

\begin{theorem}\label{c2.6} For any tame family $\K$ of compact subsets of a strongly locally homogeneous completely metrizable space $X$, any two dense $\K$-tame decompositions $\A,\mathcal B$ of $X$ are topologically equivalent. Moreover, for any open cover $\U$ of $X$ there is a homeomorphism $\Phi:X\to X$ such that $\Phi(\A)=\mathcal B$ and $(\Phi,\id_X)\prec\W$, where $$\W=\{\St(A,\U)\cup\St(B,\U):A\in\A,\;B\in\BB,\;\;\St(A,\U)\cap\St(\BB,\U)\ne\emptyset\}.$$
\end{theorem}

\section{Some properties of Menger manifolds}\label{s3}

In this section we establish some properties of Menger manifolds. Basic information on the Theory of Menger manifolds can be found in \cite[Ch.4]{Chig}. We start by recalling the necessary definitions.

Let $n\ge 0$ be any non-negative integer.
A subset $A$ of a topological space $X$ is called a {\em $UV^{n-1}$-set} for $n\in\w$ if for any open neighborhood $U$ of $A$ in $X$ there exists a neighborhood $V\subset X$ of $A$ such that every map $f:[0,1]^k\setminus(0,1)^k \to V$ from the boundary of the $k$-dimensional cube $[0,1]^k$, $k\le n$, has a continuous extension $\bar f:[0,1]^k\to U$.

A topological space $X$ is called
\begin{itemize}
\item a {\em $C^{n-1}$-space} if $X$ is a $UV^{n-1}$-set in $X$;
\item an {\em $LC^{n-1}$-space} if each singleton $\{x\}\subset X$ is a $UV^{n-1}$-set in $X$;
\item a {\em $UV^n$-compactum} if $X$ is homeomorphic to a closed $UV^n$-subset of the Hilbert cube.
\end{itemize}

It is clear that a topological space $X$ is a $C^{n-1}$-space if and only if all homotopy groups $\pi_k(X,x_0)$, $x_0\in X$, $k<n$, are trivial.

The following characterization of absolute (neighborhood) extensors in dimension $n$ is well-known and can be found in \cite[2.1.12]{Chig}.

\begin{proposition}\label{LCn} A metrizable topological space $X$ is an $\ANE[n]$-space (resp. an $\AEE[n]$-space) if and only if $X$ is an $LC^{n-1}$-space (resp. an $LC^{n-1}$-space and a $C^{n-1}$-space).
\end{proposition}

We shall need the following known Extension Property of $\ANE[n]$-spaces, which can be found in  \cite[4.1.7]{Chig}.

\begin{proposition}\label{n-homotopy} For any open cover $\U$ of a metrizable $ANE[n]$-space $X$ there is an open cover $\V$ of $X$ such that for any closed subset $B$ of a metrizable space $A$ of dimension $\dim(A)\le n$ and any maps $f:A\to X$ and $g:B\to X$ with $(g,f|B)\prec\V$ there is a map $\bar g:A\to X$ such that $\bar g|B=g$ and $(\bar g,f)\prec\U$.
\end{proposition}

A map $f:X\to Y$ between metrizable spaces is called a {\em $UV^{n-1}$-map} if for each $y\in Y$ the preimage $f^{-1}(y)$ is a $UV^{n-1}$-compactum. A surjective $UV^{n-1}$-map will be called {\em $UV^{n-1}$-surjection}. It is well-known that the Menger cube $\mu^n$ is an $UV^{n-1}$-compactum (see, e.g. Lemma 5.5 of \cite{Armen}).

The following important result can be found in \cite[4.1.20]{Chig}.

\begin{proposition}\label{near} Proper $UV^{n-1}$-surjections between $\mu^n$-manifolds are near homeomorphisms.
\end{proposition}

Another useful property of the Menger cubes is the $Z$-Set Unknotting Theorem \cite[4.1.6]{Chig}:

\begin{theorem}\label{Z-unknot} Any homeomorphism $h:A\to B$ between two $Z$-sets of the Menger cube $\mu^n$ can be extended to a homeomorphism $\bar h:\mu^n\to\mu^n$.
\end{theorem}

\begin{lemma}\label{l3.5} For each tame open ball $B$ in a $\mu^n$-manifold $M$ there is a $UV^{n-1}$-retraction $r:\bar B\to\partial B$.
\end{lemma}

\begin{proof} Since the tame closed ball $\bar B$ is homeomorphic to the Menger cube $\mu^n$, we can apply Theorem 4.3.5 of \cite{Chig} and find a $UV^{n-1}$-surjection $\pi:\bar B\to\mu^n$ and a continuous map $s:\mu^n\to \bar B$ such that $\pi\circ s$ is the identity map of $\mu^n$ and $s(\mu^n)$ is a $Z$-set in $\bar B$. It follows that $s$ is an embedding of $\mu^n$ in $\bar B$ and hence $s(\mu^n)$ is a $Z$-set of $\bar B$, homeomorphic to $\mu^n$. Since the boundary $\partial B$ of $B$ in $M$ also is a $Z$-set in $\bar B$, homeomorphic to $\mu^n$, we can apply the $Z$-Set Unknotting Theorem~\ref{Z-unknot} and find a homeomorphism of pairs $h:(\bar B,\partial B)\to (\bar B,s(\mu^n))$. Then the map $r=h^{-1}\circ s\circ \pi\circ h:\bar B\to\partial B$ is a $UV^{n-1}$-surjection, being the composition of the $UV^{n-1}$-surjection $\pi$ and the homeomorphisms $h$, $h^{-1}\circ s$. To see that $r$ is a retraction of $\bar B$ onto $\partial B$, observe that for every $x\in \partial B$ the point $h(x)$ belongs to $s(\mu^n)$ and hence $h(x)=s(y)$ for some $y\in\mu^n$. Since $y=\pi\circ s(y)=\pi\circ h(x)$, we conclude that $h(x)=s\circ \pi\circ h(x)$ and hence
$$r(x)=h^{-1}\circ s\circ \pi\circ h(x)=h^{-1}\circ h(x)=x,$$
so $r$ is a retraction of $\bar B$ onto its boundary $\partial B$.
\end{proof}

\begin{proposition}\label{p3.6} For any tame open set $U$ in an $\mu^n$-manifold $M$ the complement $M\setminus U$ is an $\mu^n$-manifold, homeomorphic to $M$.
\end{proposition}

\begin{proof} Write the tame open set $U$ as the union $U=\bigcup\mathcal B$ of a vanishing family of tame open balls with disjoint closures in $M$. By Lemma~\ref{l3.5}, for each tame open ball $B\in\mathcal B$ there exists a $UV^{n-1}$-retraction $r_B:\bar B\to\partial B$. Extend $r_B$ to the  $UV^{n-1}$-retraction $\bar r_B:M\to M\setminus B$ such that $\bar r_B|\bar B=r_B$.

The $UV^{n-1}$-retractions $r_B$, $B\in\mathcal B$, form a $UV^{n-1}$-retraction $r:M\to M\setminus M\setminus U$ defined by
$$
r(x)=\begin{cases}
r_B(x)&\mbox{ if $x\in B\in\mathcal B$}\\
x&\mbox{otherwise}.
\end{cases}
$$

\begin{claim} $M\setminus U$ is a $\mu^n$-manifold.
\end{claim}

\begin{proof} By Theorem~\ref{best2}, the $\mu^n$-manifold $M$ is an $\ANE[n]$-space, and then so is its retract $X\setminus U$. It remains to check that $M\setminus U$ has the disjoint $n$-cells property.
Fix any metric $d$ generating the topology of $M$. Given any $\epsilon>0$ and a  map $f:\II^n\times\{0,1\}\to S$ we need to find a map $\tilde f:\II^\w\times \{0,1\}\to M$ such $\tilde f(\II^n\times\{0\})\cap \tilde f(\II^n\times\{1\})=\emptyset$ and that $d(\tilde f,f)=\sup_{t\in\II^n}d(\tilde f(t),f(y))<\e$.

By the discrete $n$-cells property of the Menger manifold $M$, the map $f:\II^n\times \{0,1\}\to M\setminus U\subset M$ can be approximated by a map $g:\II^n\times \{0,1\}\to M$ such that
$d(g,f)<\frac12\epsilon$ and $g(\II^n\times\{0\})\cap g(\II^n\times\{1\})=\emptyset$. Fix a positive real number $\delta<\epsilon$ such that $$\delta\le \dist\big(g(\II^n\times\{0\}),g(\II^n\times\{1\})\big)=\inf\big\{d(x,y):x\in g(\II^n\times\{0\}),\;\;y\in g(\II^n\times\{1\})\big\}.$$

The vanishing property of the family $\mathcal B$ guarantees that the subfamily $\mathcal B'=\{B\in\BB:\diam(B)\ge\delta/5\}$ is discrete in $M$. By collectivewise normality of $M$, for each set $B\in\BB'$ its closure $\bar B$ has an open neighborhood $O(\bar B)\subset M$ such that the indexed family $\big(O(\bar B)\big)_{B\in \BB'}$ is discrete in $M$.

\begin{claim}\label{cl10.2} For every $B\in\BB'$ there is a map $g_B:\II^n\times\{0,1\}\to M\setminus B$ such that
\begin{enumerate}
\item $d(g_B,\bar r_B\circ g)<\delta/5$;
\item $g_B|g^{-1}(M\setminus O(\bar B))= g|g^{-1}(M\setminus O(\bar B))$;
\item $g_B(g^{-1}(\bar B))\subset\partial B$;
\item $g_B(g^{-1}(O(\bar B)))\subset O(\bar B)$, and
\item $g_B(\II^n\times\{0\})\cap g_B(\II^n\times\{1\})=\emptyset$.
\end{enumerate}
\end{claim}

\begin{proof} By normality of $M$, the closed set $\bar B$ has an open neighborhood $U(\bar B)\subset M$ whose closure $\bar U(\bar B)$ is contained in $O(\bar B)$. Consider the closed subset $F_B=g^{-1}(\bar B)\subset\II^n\times\{0,1\}$, and its open neighborhoods $O(F_B)=g^{-1}(O(\bar B))$ and $U(F_B)=g^{-1}(U(\bar B))$.
It follows from $\bar U(\bar B)\subset O(\bar B)$ that $\bar U(F_B)\subset O(F_B)$.

Next, consider the map $\bar r_B\circ g|O(F_B):O(F_B)\to O(\bar B)\setminus B$. By Lemma~\ref{n-homotopy},
there is an open cover $\U_B$ of $O(\bar B)\setminus B$ such that any map $g':F_B\to O(\bar B)\setminus B$ with $(g',\bar r_B\circ g|F_B)\prec\U_B$ can be extended to a map $g'_B:O(F_B)\to  O(\bar B)\setminus B$ such that $g'_B|O(F_B)\setminus U(F_B)=g|O(F_B)\setminus (F_B)$ and $d(g'_B,g|O(B))<\delta/5$.

Since the boundary $\partial B$ of the tame open ball $B$ in $M$ is homeomorphic to the Menger cube $\mu^n$, by Theorem~4.1.19 \cite{Chig}, the map $\bar r_B\circ g|F_B\to\partial B$ can be approximated by an injective map $g':F_B\to\partial B$ such that $(g',g|F_B)\prec\U_B$. By the choice of the cover $\U_B$ the map $g'$ can be extended to a continuous map $g'_B:O(F_B)\to O(\bar B)\setminus B$ such that $g'_B|O(F_B)\setminus U(F_B)=g|O(F_B)\setminus U(F_B)$ and $d(g'_B,g|O(F_B))<\delta/5$.

Extend the map $g'_B$ to a continuous map $g_B:\II^n\times\{0,1\}\to M\setminus \bigcup\BB'$ such that
$$g_B(x)=\begin{cases}
g'_B(x)&\mbox{if $x\in O(B)$}\\
g(x)&\mbox{otherwise}.
\end{cases}
$$
It is easy to see that the map $g_B$ satisfies the conditions (1)--(5).
\end{proof}

Now define a map $\tilde g:\II^n\times\{0,1\}\to M'$ by the formula
$$\tilde g(x)=\begin{cases}
g_B(x)&\mbox{if $x\in g^{-1}(O(\bar B))$ for some $B\in\BB'$};\\
g(x)&\mbox{otherwise}.
\end{cases}
$$
Claim~\ref{cl10.2} implies that $d(\tilde g,g)<\delta/5$ and $\tilde g(\II^n\times\{0\})\cap \tilde g(\II^n\times\{1\})=\emptyset$.
Finally, put $\tilde f=r\circ \tilde g:\II^n\times\{0,1\}\to M\setminus U$.

The choice of the family $\BB'$ guarantees that $d(\tilde f,\tilde g)<\delta/5$ and hence $d(\tilde f,g)<\frac25\delta$ and $d(\tilde f,f)\le d(\tilde f,g)+d(g,f)<\frac25\delta+\frac12\epsilon<\epsilon$. The choice of $\delta\le\dist\big(g(\II^n{\times}\{0\}),g(\II^n{\times}\{0\})\big)$ guarantees that
$$\dist\big(\tilde f(\II^n{\times}\{0\}),\tilde f(\II^n{\times}\{0\})\big)\ge\delta-2d(\tilde f,g)\ge\frac15\delta>0$$and thus $\tilde f(\II^n{\times}\{0\})\cap\tilde f(\II^n{\times}\{1\})=\emptyset$. By Characterization Theorem~\ref{best2}, the space $M\setminus U$ is an $\mu^n$-manifold.
\end{proof}

By Proposition~\ref{near}, the proper $UV^{n-1}$-retraction $r:M\to M\setminus U$ between the $\mu^n$-manifolds $M$ and $M\setminus U$ is a near homeomorphism, which implies that $M\setminus U$ is homeomorphic to $M$.
\end{proof}

Since each tame open ball is a tame open set, Proposition~\ref{p3.6} implies:

\begin{corollary}\label{tameball} For any tame open ball $U$ in a Menger manifold $M$, the complement $M\setminus U$ is a Menger manifold, homeomorphic to $M$.
\end{corollary}

\begin{proposition}\label{p3.9} Let $\mathcal D$ be a vanishing decomposition of a  $\mu^n$-manifold $M$ such that each set $D\in \mathcal D$ is a compact $UV^{n-1}$-set in $M$ and a $Z$-set in $M$. Then the decomposition space $X/\mathcal D$ is a $\mu^n$-manifold and the quotient map $q:M\to M/\mathcal D$ is a near homeomorphism.
\end{proposition}

\begin{proof} Since each $\mu^n$-manifolds can be decomposed into a topological sum of $\sigma$-compact $\mu^n$-manifolds, we lose no generality assuming that the $\mu^n$-manifold $M$ is $\sigma$-compact.
In this case the vanishing decomposition $\mathcal D$ has at most countable non-degeneracy part $\mathcal D^\circ$, which implies that the union $\bigcup\mathcal D^\circ$ is a $\sigma Z$-set in $M$.

By Lemma 2.2 of \cite{BR1}, the vanishing decomposition $\mathcal D$ of $X$ is upper semicontinuous, which implies that quotient map $q:X\to X/\mathcal D$ is perfect and hence is a $UV^{n-1}$-surjection. By Lemma~\ref{l2.1}, the decomposition space $Y=X/\mathcal D$ is metrizable and locally compact.

The space $X$, being a $\mu^n$-manifold, is an $LC^{n-1}$-space. Since $q:X\to Y$ is a $UV^{n-1}$-surjection, $Y$ is an $LC^{n-1}$-space according to  Proposition 2.1.32 of \cite{Chig}. By Lemma~\ref{LCn}, $Y$ is an $\ANE[n]$-space. To apply Bestvina's characterization Theorem~\ref{best2}, it remain to prove that the space $Y$ has the disjoint $n$-cells property.

For this fix an open cover $\U$ of $Y$ and two maps $f,g:\II^n\to Y$.
By the paracompactness of $Y$, there exists an open cover $\V$ of $Y$
such that $\St(\V)\prec\U$. Consider the open cover $q^{-1}(\V)=\{q^{-1}(V):V\in\V\}$ of the space $X$.
By Proposition 2.1.31 \cite{Chig}, there are maps $\tilde f,\tilde g:\II^n\to X$ such that $(q\circ\tilde f,f)\prec\V$ and $(q\circ\tilde g,g)\prec\V$. Since $X$ is a $\mu^n$-manifold and $\cup\mathcal D^\circ$ is a $\sigma Z$-set in $X$, there are two maps $\tilde f',\tilde g':\II^n\to X$ such that $(\tilde f',\tilde f)\prec q^{-1}(\V)$, $(\tilde g',\tilde g)\prec q^{-1}(\V)$, $\tilde f'(\II^k)\cap\tilde g'(\II^k)=\emptyset$ and $\tilde f'(\II^k)\cup\tilde g'(\II^k)\subset X\setminus\bigcup\mathcal D^\circ$. Then the maps $f'=q\circ\tilde f':\II^n\to Y$ and $g'=q\circ\tilde g':\II^n\to Y$ have the properties: $(f',f)\prec\St(\V)\prec\U$, $(g',g)\prec\St(\V)\prec\U$ and $f'(\II^n)\cap g'(\II^n)=\emptyset$. The latter property follows from the injectivity of the restriction $q|X\setminus\bigcup\mathcal D^\circ$. This completes the proof of the disjoint $n$-cells property of the locally compact $\ANE[n]$-space $Y$. By Theorem~\ref{best2}, the space $Y$ is an $\mu^n$-manifold and by Proposition~\ref{near}, the $UV^{n-1}$-surjection $q:X\to Y$ is a near homeomorphism.
\end{proof}

\begin{proposition}\label{shrink} A vanishing decomposition $\mathcal D$ of an $\mu^n$-manifold $M$ is shrinkable if
each non-degeneracy element $D\in\mathcal D^\circ$ is a tame closed ball in $M$.
\end{proposition}

\begin{proof} Find a disjoint family $\U$ of tame open balls in the $\mu^n$-manifold $M$ such that $\mathcal D^\circ=\{\bar B:B\in\U\}$. Then $U=\bigcup\U$ is a tame open set in $M$ and by Proposition~\ref{p3.6}, the complement $M\setminus U$ is a $\mu^n$-manifold. Moreover, by the proof of Proposition~\ref{p3.6}, there is an $UV^{n-1}$-retraction $r:M\to M\setminus U$. By the definition of a tame open ball, the boundary $\partial B$ of each tame open ball $B\in\U$ in $M$ is homeomorphic to the Menger cube $\mu^n$ and is a $Z$-set in  $M\setminus B$.

We claim that $\partial B$ is a $Z_n$-set in $M\setminus U$. Given any map $f:\II^n\to M\setminus U$ and any open cover $\W$ of $M\setminus U$, consider the open cover $r^{-1}(\W)=\{r^{-1}(W):W\in\W\}$ of the subspace $M\setminus B\subset M$. Since $\partial B$ is a $Z$-set in $M\setminus B$, there is a map $f':\II^n\to M\setminus B$ such that $f'(\II^n)\cap\partial B=\emptyset$ and $(f',f)\prec r^{-1}(\W)$. Then the map $\tilde f=r\circ f':\II^n\to M\setminus U$ has the desired properties: $(\tilde f,f)\prec\W$ and $\tilde f(\II^n)\cap \partial B=\emptyset$. The latter equality follows from the fact that $r(K)\subset\partial K$ for each $K\in \U$. So, $\partial B$ is a $Z_n$-set in $M\setminus U$.
Since $M\setminus U$ is a $\mu^n$-manifold, $\partial B$ is a $Z$-set in $M\setminus U$ according to Proposition 4.1.13 of \cite{Chig}. Now we see that $\partial\mathcal D=\{\partial B:B\in\mathcal D^\circ\}\cup\big\{\{x\}:x\in M\setminus \bigcup\mathcal D^\circ\big\}$ is a vanishing decomposition of the $\mu^n$-manifold $M\setminus U$ into $Z$-sets which are $UV^{n-1}$-sets (being homeomorphic to $\mu^n$ or singletons). By Proposition~\ref{p3.9}, the decomposition space $N=(M\setminus U)/\partial\mathcal D$ is a $\mu^n$-manifold and the quotient map $q:M\setminus U\to N=(M\setminus U)/\partial\mathcal D$ is a near homeomorphism. Now consider the perfect map $q_{\mathcal D}\circ r:M\to N$ and observe that $\mathcal D=\{(q_{\partial\mathcal D}\circ r)^{-1}(y):y\in N\}$, which implies that the decomposition space $M/\mathcal \partial \mathcal D$ is homeomorphic to the $\mu^n$-manifold $N$. Now we see that the quotient map $q_{\mathcal D}:M\to M/{\mathcal D}$, being a $UV^{n-1}$-surjection between $\mu^n$-manifolds, is a near homeomorphism.
\end{proof}

Since open subspaces of $\mu^n$-manifolds are $\mu^n$-manifolds, Proposition~\ref{shrink} has a self-generalization:

\begin{proposition}\label{strong-shrink} A vanishing decomposition $\mathcal D$ of a Menger manifold $M$ is strongly shrinkable if
each non-degeneracy element $D\in\mathcal D^\circ$ is a tame closed ball in $M$.
\end{proposition}

\section{Constructing tame balls in Menger manifolds}\label{s4}

In this section we shall construct tame balls in Menger manifolds. In particular, we shall prove Lemma~\ref{l4.2}, which implies Theorem~\ref{base}, announced in the Introduction.

First we recall a standard construction of the Menger cube $M^k_n$ where $k\ge 2n+1$. In the discrete cube $\{0,1,2\}^k$ consider the subset
$$T^k_n=\big\{(x_1,\dots,x_k)\in\{0,1,2\}^k:|\{i\in \{1,\dots,k\}:x_i=1\}|\le n\big\}.$$
Let $p_k:\{0,1,2\}^k\to\{0,1,2\}$, $p_k:(x_1,\dots,x_k)\mapsto x_k$, be the projection onto the last coordinate. It follows that for any $y\in\{0,2\}$ we get
\begin{equation}
p_k^{-1}(y)\cap T^k_n=T^{k-1}_n\times \{y\}.
\end{equation}

By definition, the Menger cube $M^k_n$ is the image of the countable product $(T^n_k)^{\IN}$ under the continuous map $$s:(T^n_k)^{\IN}\to[0,1]^k, \;\;s:(x_i)_{i=1}^\infty\mapsto\sum_{i=1}^\infty\frac{x_i}{3^i}.$$

Bestvina's Characterization Theorem~\ref{best1} implies that for any $n\ge 0$ and $k\ge 2n+1$ the Menger cube $M^k_n$ is homeomorphic to the Menger cube $\mu^n=M^{2n+1}_n$, see \cite[p.98]{Best}. For $n=0$ and $k=1$ the Menger cube $\mu^0=M^1_0$ coincides with the standard Cantor set on $[0,1]$.

Now we shall establish some properties of sections of the Menger cube $M_n^k$ by hyperplanes.
Consider the projection $\pr_k:[0,1]^k\to [0,1]$, $\pr_k:(x_1,\dots,x_k)\mapsto x_k$, of the cube $[0,1]^k$ onto the last coordinate. For two points $a<b$ in the unit interval $[0,1]$, consider the retraction $r_{[a,b]}:[0,1]\to [a,b]$ such that $r(x)=x$ for all $x\in[a,b]$, $r([0,a])=\{a\}$ and $r([b,1])=\{b\}$. This retraction induces the retraction $\bar r:[0,1]^k\to\pr_k^{-1}([a,b])$ defined by $\bar r:(x_1,\dots,x_k)\mapsto (x_1,\dots,x_{k-1},r(x_k))$.

\begin{lemma}\label{l4.1} If $k\ge 2n+2$ and $a<b$ are any points of the Cantor set $M^1_0\subset[0,1]$, then:
\begin{enumerate}
\item the set $M^k_n\cap\pr_k^{-1}(a)$ is equal to $M^{k-1}_n\times\{a\}$ and hence is homeomorphic to the Menger cube $\mu^n$;
\item the restriction $\bar r_{[a,b]}|M^k_n:M^k_n\to \pr_k^{-1}([a,b])$ is a retraction of $M^k_n$ onto $M^k_n\cap \pr^{-1}_k([a,b])$;
\item $M^k_n\cap \pr^{-1}_k(a)$ is a $Z$-set in $M^k_n\cap\pr_k^{-1}([a,b])$ if $a$ is a non-isolated point of $[a,b]\cap M^1_0$;
\item $M^k_n\cap \pr^{-1}_k(b)$ is a $Z$-set in $M^k_n\cap\pr_k^{-1}([a,b])$ if $b$ is non-isolated point of $[a,b]\cap M^1_0$;
\item $M^k_n\cap \pr^{-1}_k([a,b])$ is homeomorphic to the Menger cube $\mu^n$ if $a,b$ are non-isolated points in $[a,b]\cap M^1_0$.
\end{enumerate}
\end{lemma}

\begin{proof} 1. The first statement follows from the equality $T^{k-1}_n\times\{0,2\}=T^k_n\cap (\{0,1,2\}^{k-1}\times\{0,2\})$.
\smallskip

2. The second statement follows from the observation that for each sequence $(x(1),\dots,x(k))\in T^k_n$ the sequences $(x(1),\dots,x(k-1),0)$ and $(x(1),\dots,x(k-1),2)$ belong to $T^k_n$.
\smallskip

3. To prove the third statement, assume that $a$ is a non-isolated point of the set $M^1_0\cap[a,b]$.
We need to check that $M^k_n\cap \pr^{-1}_k(a)$ is a $Z$-set in $M^k_n\cap\pr_k^{-1}([a,b])$. Fix an open cover $\U$ of $M^k_n\cap\pr_k^{-1}([a,b])$ and find $a'\in (a,b]\cap M^1_0$ so close to $a$ that the retraction $f=\bar r_{[a',b']}|M^k_n\cap\pr^{-1}_k([a,b]):M^k_n\cap\pr_k^{-1}([a,b])\to M^k_n\cap\pr^{-1}_k([a',b])$ is $\U$-near to the identity map of $M^k_n\cap\pr_k^{-1}([a,b])$. Since $\bar r_{[a',b]}(M^k_n\cap\pr_k^{-1}([a,b])\cap \pr^{-1}_k(a)=\emptyset$, the retraction $f$ witnesses that $M^k_n\cap\pr^{-1}_k(a)$ is a $Z$-set in $M^k_n\cap\pr^{-1}([a,b])$.
\smallskip

4. The fourth statement can be proved by analogy with the third one.
\smallskip

5. The fifth statement can be derived from the preceding statements with help of the characterization Theorem~\ref{best1}.
\end{proof}

\begin{lemma}\label{l4.2} Each point of the Menger cube $\mu^n$ has a neighborhood base consisting of tame open balls.
\end{lemma}

\begin{proof} Fix any integer $k\ge 2n+2$ and consider the Menger cube $M^k_n$, which is homeomorphic to $\mu^n$ by \cite[p.98]{Best}. By Lemma~\ref{l4.1}, the set $K=M^{k}_n\cap \pr^{-1}_k(0)$ is a $Z$-set in $M^k_n$, homeomorphic to $\mu^n$. By Proposition~\ref{p3.9}, the quotient map $q:M^k_n\to M^k_n/K$ is a near homeomorphism and hence the quotient space $M^k_n/K$ is homeomorphic to the Menger cube $\mu^n$.
Let $E$ denote the set of points $a$ of the Cantor cube $M^1_0$ such that $a$ is a non-isolated point in the sets $[0,a]\cap M^1_0$ and $[a,1]\cap M^1_0$. It is clear that $E$ is a dense subset of $M^1_0$.

By Lemma~\ref{l4.1}, for every point $\e\in E$, the set $U_\e=M^k_n\cap\pr^{-1}([0,\e))$ is a tame open ball in $M^k_n$. Using Proposition~\ref{p3.9}, it can be shown that the set $V_\e=q(U_\e)=U_\e/K$ is a tame open ball in $M^k_n/K$. It remains to observe that $\{V_\e:\e\in E\}$ is a neighborhood base of the point $\{K\}\in M^k_n/K$ consisting of tame open balls in the space $M^k_n/K$, which is homeomorphic to the Menger cube $\mu^n$. The topological homogeneity of $\mu^n$ implies that each point of $\mu^n$ has a neighborhood base consisting of tame open balls.
\end{proof}

\begin{lemma}\label{l:K-tame} The family $\mathcal K$ of tame closed balls in an $\mu^n$-manifold $M$ is tame.
\end{lemma}

\begin{proof} By Definition~\ref{d:K-tame} of a tame family, we need to check four conditions.
\smallskip

1. It follows from the definition of a tame closed ball that the family $\mathcal K$ is topologically invariant.
\smallskip

2. Next, we should check that each tame closed ball $B\in\K$ is locally shrinkable. By \cite[p.42]{Dav}, this is equivalent to saying that the quotient map $q_B:M\to M/B$ is a strong near homeomorphism. But this follows from Proposition~\ref{strong-shrink}.
\smallskip

3. To prove that the family $\K$ of tame closed balls has the local shift property, fix a point $x\in M$ and a neighborhood $O_x$ of $x$. By Lemma~\ref{l4.2}, the neighborhood $O_x$ contains the closure $\bar U_x$ of some tame open ball $U_x$. We claim  that for any two tame closed balls $\bar A, \bar B\subset U_x$ there is a homeomorphism $\bar h:M\to M$ such that $\bar h(\bar A)=\bar B$ and $\bar h|M\setminus O_x=\id$.
Here $A,B$ are the interiors of the tame closed balls $\bar A$ and $\bar B$ in $M$.
By Corollary~\ref{tameball}, the complements $\bar U_x\setminus A$ and $\bar U_x\setminus B$ are homeomorphic to the  Menger cube $\mu^n$. Since the sets $\partial U_x$, $\partial A$, $\partial B$ are $Z$-sets in the tame closed ball $\bar U_x$, the $Z$-Unknotting Theorem~\ref{Z-unknot} allows us to find a homeomorphism $h:\bar U_x\setminus A\to\bar U_x\setminus  B$ such that $h|\partial U_x=\id$ and $h(\partial A)=\partial B$. Since the tame closed balls $\bar A$ and $\bar B$ are homeomorphic to $\mu^n$, the $Z$-Set Unknotting Theorem~\ref{Z-unknot} guarantees that the homeomorphism $h|\partial A:\partial A\to\partial B$ extends to a homeomorphism $\tilde h:\bar A\to \bar B$. Then the homeomorphism $\bar h:M\to M$ defined by $\bar h|\bar U_x\setminus A=h$, $\bar h|\bar A=\tilde h$ and $\bar h|M\setminus U_x=\id$ has the desired properties: $\bar h(\bar A)=\bar B$ and $\bar h|M\setminus O_x=\id$. This completes the proof of the local shift property of the family $\K$.
\smallskip

4. Lemma~\ref{l4.2} implies that each non-empty open set contains a tame closed ball.
\end{proof}

\begin{proposition}\label{p4.3} Any two dense disjoint vanishing families $\A,\mathcal B$ of tame closed balls in a Menger manifold $M$ are topologically equivalent. Moreover, for any open cover $\U$ of $M$ there is a homeomorphism $\Phi:X\to X$ such that $\Phi(\A)=\mathcal B$ and $(\Phi,\id_X)\prec\W$, where $$\W=\{\St(A,\U)\cup\St(B,\U):A\in\dot\A,\;B\in\dot\BB,\;\;\St(A,\U)\cap\St(\BB,\U)\ne\emptyset\},$$
$\dot\A=\A\cup\big\{\{x\}:x\in M\setminus \textstyle{\bigcup}\A\big\}$ and $\dot{\mathcal B}={\mathcal B}\cup\big\{\{x\}:x\in M\setminus \textstyle{\bigcup}\mathcal B\big\}$.
\end{proposition}

\begin{proof} The vanishing families $\A,\mathcal B$ can be completed to vanishing decompositions $\dot\A=\A\cup\big\{\{x\}:x\in M\setminus\bigcup\A\big\}$ and $\dot{\mathcal B}=\mathcal B\cup\big\{\{x\}:x\in M\setminus\bigcup\mathcal B\big\}$ of the Menger manifold $M$. By Corollary 4.1.17 of \cite{Chig}, the Menger manifold $M$ is strongly topologically homogeneous. By Lemma~\ref{l:K-tame}, the family $\K$ of tame closed balls in $M$ is tame. By Proposition~\ref{strong-shrink}, the vanishing decompositions $\dot\A$ and $\dot{\mathcal B}$ are strongly shrinkable and hence are $\K$-tame according to Definition~\ref{d:d-Ktame}. Now Proposition~\ref{p4.3} follows from Theorem~\ref{c2.6}.
\end{proof}

\section{Proof of Theorem~\ref{t3}}\label{s5}

Given two dense tame $G_\delta$-sets $G,G'$ in a Menger manifold $M$, and an open cover $\U$ of $M$ we need to find a homeomorphism $h:(M,G)\to(M,G')$, which is $\U$-near to the identity homeomorphism of $M$.

Depending on the dimension of the Menger manifold $M$, two cases are possible.

If $\dim(M)=0$, then $M$ is a manifold modeled on the Cantor cube $\mu^0$. It follows from the definition of a tame $G_\delta$-set that the complements $M\setminus G$ and $M\setminus G'$ are  everywhere uncountable in $M$ in the sense that they have uncountable intersections with each non-empty open subset of $M$. This fact can be used to show that $M\setminus G=\bigcup_{i\in\w}Z_i$ and for an increasing sequence $(Z_i)_{i\in\w}$ of compact subsets without isolated points in $M$ such that each
set $Z_i$ is nowhere dense in $Z_{i+1}$. By analogy we can represent $M\setminus G'=\bigcup_{i\in\w}Z_i'$. Now the standard technique of skeletoids \cite[\S IV,V]{BP} \cite[\S2.2.2]{Chig} allows us to construct a homeomorphism $h:M\to M$ such that $(h,\id)\prec\U$ and $h(M\setminus G)=M\setminus G'$.

If $\dim(M)>0$, then the homeomorphism $h(M,G)\to (M,G')$ can be constructed by analogy with the proof of Theorem 3 in \cite{BR2}.

\end{document}